\documentclass[11pt]{amsart}
\usepackage{amsmath,amsfonts,mathrsfs}
\usepackage{amssymb}

\usepackage[unicode,bookmarks]{hyperref}
\usepackage[usenames,dvipsnames]{xcolor}
\hypersetup{colorlinks=true,citecolor=NavyBlue,linkcolor=Brown,urlcolor=Orange}

\usepackage{tikz-cd,adjustbox}
\usepackage[arrow,curve,matrix]{xy}

\newif\ifdraft
\draftfalse

\counterwithin{equation}{subsection}
\counterwithout{subsection}{section}
\counterwithin{figure}{subsection}

\newtheorem{theorem}[equation]{Theorem}
\newtheorem*{theorem*}{Theorem}
\newtheorem{lemma}[equation]{Lemma}
\newtheorem*{lemma*}{Lemma}
\newtheorem{corollary}[equation]{Corollary}

\newtheorem*{proposition*}{Proposition}

\theoremstyle{definition}
\newtheorem{definition}[equation]{Definition}
\newtheorem*{definition*}{Definition}
\newtheorem{remark}[equation]{Remark}

\newtheorem{example}[equation]{Example}
\newtheorem*{example*}{Example}
\newtheorem*{problem*}{Problem}

\theoremstyle{plain}

\newcommand{\m}{\mathfrak m}

\newcommand{\C}{\mathcal C}

\newcommand{\E}{\mathcal E}

\renewcommand{\H}{\mathcal H}

\renewcommand{\O}{\mathcal O}

\newcommand{\bL}{{\bf L}}

\newcommand{\bR}{{\bf R}}

\renewcommand{\AA}{\mathbb A}
\newcommand{\CC}{\mathbb C}
\newcommand{\DD}{\mathbb D}

\newcommand{\HH}{\mathbb H}

\newcommand{\PP}{\mathbb P}
\newcommand{\QQ}{\mathbb Q}

\newcommand{\ZZ}{\mathbb Z}

\newcommand{\sE}{\mathscr E}
\newcommand{\sF}{\mathscr F}

\DeclareMathOperator{\Sym}{Sym}

\DeclareMathOperator{\coker}{coker}

\newcommand{\xto}{\xrightarrow} 

\newcommand{\RHom}{{\bf R} \H om}
\newcommand*{\sHom}{\mathscr{H}\kern -.5pt om}
\DeclareMathOperator{\Ext}{Ext}
\newcommand{\sExt}{\E xt} 

\newcommand{\DB}{\underline{\Omega}} 

\DeclareMathOperator{\Gr}{Gr}
\DeclareMathOperator{\DR}{DR}

\DeclareMathOperator{\tors}{tors}

\DeclareMathOperator{\lcdef}{lcdef}

\DeclareMathOperator{\Spec}{Spec}

\DeclareMathOperator{\Pic}{Pic}

\DeclareMathOperator{\lcd}{lcd}
\DeclareMathOperator{\codim}{codim}

\DeclareMathOperator{\depth}{depth}

\newcommand{\theoremref}[1]{\hyperref[#1]{Theorem~\ref*{#1}}}
\newcommand{\lemmaref}[1]{\hyperref[#1]{Lemma~\ref*{#1}}}
\newcommand{\definitionref}[1]{\hyperref[#1]{Definition~\ref*{#1}}}
\newcommand{\propositionref}[1]{\hyperref[#1]{Proposition~\ref*{#1}}}
\newcommand{\conjectureref}[1]{\hyperref[#1]{Conjecture~\ref*{#1}}}
\newcommand{\corollaryref}[1]{\hyperref[#1]{Corollary~\ref*{#1}}}
\newcommand{\exampleref}[1]{\hyperref[#1]{Example~\ref*{#1}}}

\makeatletter
\let\old@caption\caption
\renewcommand*{\caption}[1]{%
	\setcounter{figure}{\value{equation}}%
	\stepcounter{equation}%
	\old@caption{#1}\relax%
}
\makeatother

\newcounter{intro}

\newtheorem{intro-conjecture}[intro]{Conjecture}
\newtheorem{intro-corollary}[intro]{Corollary}
\newtheorem{intro-theorem}[intro]{Theorem}

\begin{document}

\title[Du Bois complexes of cones over singular varieties]{Du Bois complexes of cones over singular varieties, local cohomological dimension, and $K$-groups}

\author[M.~Popa]{Mihnea~Popa}
\address{Department of Mathematics, Harvard University, 
1 Oxford Street, Cambridge, MA 02138, USA} 
\email{{\tt mpopa@math.harvard.edu}}

\author[W.~Shen]{Wanchun~Shen}
\address{Department of Mathematics, Harvard University, 
1 Oxford Street, Cambridge, MA 02138, USA} 
\email{{\tt wshen@math.harvard.edu}}

\thanks{MP was partially supported by NSF grant DMS-2040378.}
\date{\today}

\dedicatory{Dedicated to the memory of Lucian B\u adescu, with gratitude.}

\subjclass[2020]{}

\begin{abstract}
We compute the Du Bois complexes of abstract cones over singular varieties, and use this to describe the local cohomological dimension
and the non-positive $K$-groups of such cones. 
\end{abstract}

\maketitle

\makeatletter
\newcommand\@dotsep{4.5}
\def\@tocline#1#2#3#4#5#6#7{\relax
  \ifnum #1>\c@tocdepth 
  \else
    \par \addpenalty\@secpenalty\addvspace{#2}%
    \begingroup \hyphenpenalty\@M
    \@ifempty{#4}{%
      \@tempdima\csname r@tocindent\number#1\endcsname\relax
    }{%
      \@tempdima#4\relax
    }%
    \parindent\z@ \leftskip#3\relax
    \advance\leftskip\@tempdima\relax
    \rightskip\@pnumwidth plus1em \parfillskip-\@pnumwidth
    #5\leavevmode\hskip-\@tempdima #6\relax
    \leaders\hbox{$\m@th
      \mkern \@dotsep mu\hbox{.}\mkern \@dotsep mu$}\hfill
    \hbox to\@pnumwidth{\@tocpagenum{#7}}\par
    \nobreak
    \endgroup
  \fi}
\def\l@section{\@tocline{1}{0pt}{1pc}{}{\bfseries}}
\def\l@subsection{\@tocline{2}{0pt}{25pt}{5pc}{}}
\makeatother



\section{Introduction}
As an undergraduate, the first author was introduced by Lucian B\u adescu to the beautiful world of algebraic geometry. One of his favorite topics was that of Lefschetz-type theorems, in particular Barth's celebrated result for arbitrary smooth subvarieties in projective space; he also often brought up Hartshorne and Ogus' work on local cohomological dimension and its connections to the topology of projective varieties. We dedicate this note on a related circle of ideas to him, in the belief that he would have enjoyed it! 

Concretely, in this paper we compute the Du Bois complexes of cones over arbitrary subvarieties in projective space; the case of smooth varieties can be found in \cite{SVV} and the references therein. We use this for several applications. 
The most important is not a new result, at least in the setting of classical cones, but we believe it nicely illustrates how such Du Bois complexes can be used in an almost algorithmic fashion:  using the  birational characterization of local cohomological dimension provided in \cite{MP-LC}, expressed in an equivalent form in terms of the depth of Du Bois complexes, we give an alternative approach to Ogus' \cite{Ogus} interpretation of the local cohomological dimension at the vertex of the cone over a projective variety $X$ in terms of the topology of $X$. As a new result, we extend this characterization to all abstract cones over $X$. In a different direction, feeding our computation into the main technical result of \cite{CHWW}, we describe the non-positive $K$-groups of abstract cones over a projective variety in terms of its Du Bois complexes, extending the result in \emph{loc. cit.} to singular complex varieties.

Let $X \subseteq \PP^N$ be a complex projective variety of dimension $n$ and codimension $r = N - n$, and let $Z = C(X)  \subseteq \AA^{N+1}$ be the affine cone over $X$. We denote by $ \lcd (Z, \AA^{N+1})$ the local cohomological dimension of $Z$ in $\AA^{N+1}$, and call 
$${\rm lcdef} (Z) : = \lcd (Z, \AA^{N+1}) - r$$
the \emph{local cohomological defect} of $Z$. This invariant depends only on $Z$, and not on the embedding in $\AA^{N+1}$; see Section \ref{scn:lcd} for details. The following is a slight reformulation of Ogus' result \cite[Theorem 4.4]{Ogus} (also recovered more recently in \cite{HP} in the smooth case, and in \cite{RSW} in general):

\begin{intro-theorem}\label{thm:main}
For an integer $c \ge 0$,  the following are equivalent:
	\begin{enumerate}
		\item ${\rm lcdef} (Z)\le c$
		\item ${\rm lcdef} (X)\le c$ and the restriction maps $H^i(\PP^N, \CC) \to H^i(X, \CC)$ are isomorphisms for $i\le n-1 - c$.
\end{enumerate}
In particular:
$${\rm lcdef} (Z) = 0 \iff {\rm lcdef} (X) = 0 ~{\rm and}~ H^i (\PP^N, \CC) \overset{\simeq}{\longrightarrow}  H^i (X, \CC) \,\,{\rm for ~all} \,\,i \le n-1.$$
\end{intro-theorem}

Note that in \cite{Ogus}, part (1) is phrased in terms of the cohomological dimension ${\rm cd} (\PP^N \smallsetminus X)$, but it is well known via a standard argument that $\lcd (Z, \AA^{N+1}) = {\rm cd} (\PP^N \smallsetminus X) + 1$.

The interest in the last equivalence in Theorem \ref{thm:main} stems from the fact that both sides are satisfied when $X$ is a smooth complete intersection in $\PP^N$. It is therefore intimately related to Hartshorne's conjecture on small codimension subvarieties.

\noindent
\emph{Note.} Recall that Barth's theorem (see \cite[Scn.3.2]{Lazarsfeld} for a survey) states that if $X$ is smooth, then we always have
$$H^i (\PP^N, \CC) \overset{\simeq}{\longrightarrow}  H^i (X, \CC) \,\,\,\,\,\,{\rm for ~all} \,\,\,\,i \le n-r.$$
According to Theorem \ref{thm:main}, this is equivalent to the universal bound ${\rm lcdef} (Z) \le r-1$. 

We provide in fact an extension of Ogus' result to any abstract cone over $X$, defined by the choice of ample line bundle $L$, in terms
of the action of $c_1 (L)$ on the singular cohomology of $X$; see Theorem \ref{thm:main-general} for the precise statement.
In our approach, this result, hence also Theorem \ref{thm:main}, follows from two main points. One is a reformulation of the characterization of ${\rm lcdef} (Z)$ in \cite[Theorem E]{MP-LC}, for any complex variety $Z$, in terms of the Du Bois complexes of $Z$:
$${\rm lcdef} (Z) = \dim Z -  \underset{k \ge 0}{\min}~ \{ \depth \DB_Z^k + k \}.$$
See Section \ref{scn:lcd} for details. The second, and the main technical result here, is the computation of the Du Bois complexes of (abstract) cones over arbitrary subvarieties in $\PP^N$, and of their depth. 

\begin{intro-theorem}\label{thm:DB}
Let $X$ be a projective variety, endowed with an ample line bundle $L$. Let
	\[Z=C(X,L)=\Spec \big(\bigoplus_{m\ge 0} H^0(X,L^m)\big)\]
	be the affine cone over $X$ associated to $L$, with cone point $x\in Z$. Then the Du Bois complexes $\DB_Z^k$ can be computed explicitly in terms of the Du Bois-Hodge theory of $X$ and $L$; see Theorem \ref{thm:DB-cones}(1). The same holds for $\depth_x \DB_Z^k$; see Theorem \ref{thm:DB-cones}(2).
\end{intro-theorem}

On a different note, to prove Theorem \ref{thm:main-general} we also need a general result of independent interest, namely a ``dual" Nakano-type vanishing theorem for Du Bois complexes. We establish this in Section \ref{scn:vanishing}, as a  simple application of the Kodaira-Saito vanishing theorem for mixed Hodge modules. 

\begin{intro-theorem}\label{thm:intro-vanishing}
Let $X$ be a complex projective variety of dimension $n$, and $L$ an ample line bundle on $X$. Then 
$$\HH^q (X, \DB_X^p \otimes L^{-1}) = 0 \,\,\,\,\,\,{\rm for ~all} \,\,\,\, p + q <  n - {\rm lcdef}(X).$$
\end{intro-theorem} 

\smallskip

Finally, in a different direction, Theorem \ref{thm:DB} is used in Section \ref{scn:K} to compute the non-positive $K$-groups of cones over a projective variety in terms of its Du Bois complexes, extending for complex varieties the result in \cite{CHWW} to abstract cones over singular varieties; this is simply an appendix to \emph{loc. cit.}, as we are still using the main technical result of that paper.

\begin{intro-corollary}\label{cor:K-groups}
Let $X$ be an $n$-dimensional complex projective variety endowed with an ample line bundle $L$, and let $Z = C(X, L)$ be the affine cone over $X$ associated to $L$. Then:
\begin{enumerate}
	\item $K_0(Z)\simeq \ZZ\oplus \bigoplus_{i=1}^n \bigoplus_{m\ge 1} \HH^i(X, \DB_{X/\QQ}^i\otimes L^m)$;
	\smallskip
	\item $K_{-\ell}(Z)\simeq \bigoplus_{i=0}^{n-\ell}\bigoplus_{m\ge 1} \HH^{\ell +i}(X,\DB_{X/\QQ}^i\otimes L^m)$ for all $\ell \ge 1$.
\end{enumerate}
\end{intro-corollary}

Here $\DB_{X/\QQ}^i$ are the Du Bois complexes over $\QQ$ of the complex variety $X$; see the beginning of Section \ref{scn:K}.
When $X$ is embedded in $\PP^N$, with $L = \O_X(1)$, the method applies to the classical cone $C(X)$ as well. For $l\ge 1$ we have $K_{-l}(C(X))=K_{-l}(C(X, L))$, while for $l=0$ we have
\[K_0(C(X))\simeq \ZZ\oplus \Pic(C(X))\oplus \bigoplus_{i=1}^n \bigoplus_{m\ge 1} \HH^i(X, \DB_{X/\QQ}^i\otimes L^m).\] 
Here $\Pic(C(X))\simeq R^+ / R$, where $R$ is the homogeneous coordinate ring of $X$, and 
$$R^+ \simeq \bigoplus_{m\ge 0} \HH^0(X,\DB_X^0\otimes L^m)$$
is its seminormalization. A few things can also be said about positive $K$-groups, but a precise description can only be achieved in special cases.

\medskip

\noindent
{\bf Acknowledgements.} 
We thank Elden Elmanto, Shiyue Li, Mircea Musta\c t\u a, Sung Gi Park, Anda Tenie and Duc Vo for helpful conversations.

\section{Main results}

\subsection{Preliminaries}
We start by collecting the main technical definitions and facts used in this paper.

\smallskip

\noindent
{\bf Du Bois complexes.}
Let $X$ be a complex algebraic variety. The \emph{filtered de Rham complex} $(\DB_X^\bullet, F)$ is an object in the bounded derived category of filtered differential complexes on $X$, introduced by Du Bois in \cite{DuBois}, following the ideas of Deligne, and intended as a replacement for the standard de Rham complex on smooth varieties.   For each $k \ge 0$, the (shifted) associated graded quotient
\[\DB^k_X : = \Gr^k_F \DB^\bullet_X[k],\]
is an object in $D^b_{\rm coh}(X)$, called the \emph{$k$-th Du Bois complex} of $X$. It follows from definition that
\[\DB_X^k \simeq \bR \epsilon_{\bullet *} \Omega_{X_{\bullet}}^k,\]
where $\epsilon_{\bullet} \colon X_{\bullet}\to X$ is a hyperresolution of $X$. 

Besides \cite{DuBois}, we refer for instance to \cite[Chapter V]{GNPP} or to \cite[Chapter 7.3]{PS} for the construction of hyperresolutions, and for a detailed treatment of Du Bois complexes. Here we only collect the main properties that will be needed in this paper; these will be used freely in the most obvious instances.  

\begin{lemma}\label{general-DB}
(1) We have $\DB_X^k|_U \simeq  \DB_U^k$ for any open subset $U\subset X$, and any morphism $f: X\to Y$ induces a canonical morphism $\DB_Y^k\to \bR f_*\DB_X^k$; see  \cite[Section 3]{DuBois} and \cite[Chapter 7.3]{PS}.

\noindent
(2) For each $k\ge 0$, there is a canonical morphism $\Omega_X^k \to \DB_X^k$, which is an isomorphism if $X$ is smooth; here 
$\Omega_X^k$ are the sheaves of K\"ahler differentials on $X$; see \cite[Section 4.1]{DuBois} or \cite[p.175]{PS}.

\noindent
(3) If $X$ is projective, there exists a Hodge-to-de Rham spectral sequence 
$$E^{p,q}_1 = \HH^q (X, \DB_X^p) \implies H^{p + q} (X, \CC),$$
which degenerates at $E_1$; see \cite[Theorem 4.5(iii)]{DuBois} or \cite[Proposition 7.24]{PS}.

\noindent
(4) Given a discriminant square, i.e. a Cartesian square		
		\[\begin{tikzcd}
			E\ar[r]\ar[d,"g"] &Y\ar[d,"f"] \\
			Z\ar[r] &X
		\end{tikzcd}\]
		such that the horizontal maps are inclusions and  $f: Y \smallsetminus E\to X \smallsetminus Z$ is an isomorphism, then for each $k \ge 0$, we have an exact triangle 
		\[\DB_X^k\to \DB_Z^k\oplus \bR f_*\DB_Y^k\to \bR g_*\DB_E^k\xto{+1} .\]
		See  \cite[Proposition 4.11]{DuBois} or \cite[Example 7.25]{PS}.
			
\noindent	
(5) Let $X_{\bullet}\to X$ be a hyperresolution. Then there is a spectral sequence
		\[E_1^{ij}=H^j(X_i, \Omega_{X_i}^k)\implies \HH^{i+j}(X,\DB_X^k).\]
See \cite[Chapter V, Lemma 3.1]{GNPP}.

\end{lemma}

\noindent
{\bf Depth of objects in the derived category.}
The notion of depth of a module has a natural extension to complexes. Let $(R, \frak{m})$ be a noetherian local ring endowed with a dualizing complex $\omega_R^\bullet$ (for us this will always be the local ring of $X$ at a closed point), and let $C$ be an element of the bounded derived category of finitely generated $R$-modules. Then one defines $\depth(C )$ in any of the following equivalent ways: 
\begin{enumerate}
\item $\min~  \{i ~|~\Ext^{-i}_R (C,\omega_R^{\bullet})\neq 0\}$;
\item $\min~  \{i ~|~\Ext^i_R(R/\m,C)\neq 0\}$;
\item $\min~  \{i ~|~H^i_{\frak{m}}(C)\neq 0\}$,
\end{enumerate}
with the convention that the depth is $-\infty$ if $C = 0$. The last two properties are studied and shown to be equivalent in \cite{FY}, while
the equivalence with the first follows from local duality.  

If $X$ is a variety and $\C$ is an element in $D^b_{\rm coh} (X)$, we will especially use the first and third interpretation, in the sense that
if $x\in X$ is a closed point, we have 
\begin{equation}\label{eqn:depth}
\depth_x (\C) = \min~  \{i ~|~\sExt^{-i}_{\O_{X, x}} (\C_x,\omega_{X,x}^{\bullet})\neq 0\} = \min~  \{i ~|~\HH^i_{x}(\C)\neq 0\}.
\end{equation}
We also set 
$$\depth(\C) : = \min_{x \in X}~ \depth(\C_x),$$
where the minimum is taken over the closed points of $X$.

\subsection{Local cohomological dimension via Du Bois complexes}\label{scn:lcd}
Let $X$ be a complex variety. A cohomological invariant that is now understood to figure prominently in the study of the Du Bois complexes 
of $X$ is the local cohomological dimension. If $Y$ is a smooth variety containing $X$ (locally), this can be seen as 
$$\lcd (X, Y) := \max ~\{q ~|~ \H^q_X \O_Y \neq 0\},$$
where the sheaf in the parenthesis is the $q$-th local cohomology sheaf of $\O_Y$ along $X$; see \cite{MP-LC} for an extensive discussion. It is in fact known that the sheaves $\H^q_X \O_Y$ are non-zero precisely in the range $\codim_Y X \le q \le \lcd (X, Y)$.

It follows from the main results in \cite{Ogus} or \cite{MP-LC} that $\dim Y - \lcd (X, Y)$ 
depends only on $X$, and not the choice of $Y$; for instance, in the language of \cite[Theorem 2.13]{Ogus}, it is equal to the de Rham depth of $X$. This provides a convenient invariant for our present purposes.

\begin{definition}
The \emph{local cohomological defect} ${\rm lcdef} (X)$ of $X$ is 
$${\rm lcdef} (X) : = \lcd (X, Y) - \codim_Y X.$$
It measures how far $X$ is ``numerically" from being a local complete intersection, and again depends only on $X$, 
and not on the choice of $Y$; indeed, we have 
$${\rm lcdef} (X) = \dim X - \big(\dim Y - \lcd (X, Y)\big).$$ 
This also shows that $\dim X \ge {\rm lcdef} (X) \ge 0$.
\end{definition}

The key result we use here is \cite[Theorem E]{MP-LC}, characterizing the local cohomological dimension in terms of the vanishing of higher direct images of forms with log poles on a log resolution. Via results of Steenbrink, this was equivalently phrased in terms of Du Bois complexes as follows: 

\begin{theorem}[{\cite[Corollary 12.6]{MP-LC}}]\label{thm:LCD}
Let $X$ be a subvariety of a smooth variety $Y$. For any integer $c$ we have 
$$\lcd(X, Y) \le c  \iff \sExt^{j + k + 1}_{\O_Y}(\DB_X^k,\omega_Y) = 0 \,\,\,\,{\rm for ~all~}j \ge c \,\,\,\,{\rm and}\,\,\,\, k \ge 0.$$
or equivalently 
$${\rm lcdef} (X) \le c  \iff  \sExt^{j + k + 1}_{\O_X}(\DB_X^k,\omega_X^{\bullet}) = 0 \,\,\,\,{\rm for ~all~}j \ge c - \dim X \,\,\,\,{\rm and}\,\,\,\, k \ge 0.$$
\end{theorem}

The second equivalent statement is simply obtained by translating into the language introduced here, and using Grothendieck duality for the inclusion $X \hookrightarrow Y$. The fundamental consequence is the following:

\begin{corollary}\label{cor:depth-DB-complex}
We have the identity 
$${\rm lcdef} (X) = \dim X -  \underset{k \ge 0}{\min}~ \{ \depth \DB_X^k + k \}.$$
\end{corollary}
\begin{proof}
Using Theorem \ref{thm:LCD}, the identity follows from ($\ref{eqn:depth}$). 
\end{proof}

This formula opens the door to the study of ${\rm lcdef} (X)$ based on objects in the bounded derived category of coherent sheaves, providing an alternative 
to the topological approach in \cite{Ogus} (or in \cite{RSW} and \cite{BBLSZ}).

\begin{example}[{\bf Varieties with ${\rm lcdef}(X)= 0$}]\label{def=0}
An important point is to recognize those varieties that behave numerically like local complete intersections, without necessarily being so. Here this is encoded in 
the condition ${\rm lcdef}(X)= 0$, or equivalently $\lcd (X, Y) = \codim_Y X$ in any embedding. Besides local complete intersections,  this is known to hold for instance
when $X$ has quotient singularities \cite[Corollary 11.22]{MP-LC}, when $X$ is affine with Cohen-Macaulay Stanley-Reisner coordinate algebra 
\cite[Corollary 11.26]{MP-LC}, when $X$ is an arbitrary Cohen-Macaulay surface or threefold by \cite[Remark p.338-339]{Ogus} and 
\cite[Corollary 2.8]{DT} respectively, and when $X$ is a Cohen-Macaulay fourfold whose local analytic Picard groups are torsion \cite[Theorem 1.3]{DT}.  
\end{example}

\subsection{Du Bois complexes of cones over singular varieties}
The set-up for this entire section and the next is the following: $X$ is a projective variety of dimension $n$, endowed with an ample line bundle $L$, and 
	\[Z=C(X,L)=\Spec \big(\bigoplus_{m\ge 0} H^0(X,L^m)\big)\] 
	is the affine cone over $X$ associated to $L$, with cone point $x\in Z$. 
		
Our first goal is to describe the Du Bois complexes $\DB_Z^k$, generalizing the result obtained in \cite[Appendix A]{SVV} (and earlier in different language in \cite{CHWW}) when $X$ is smooth.  More precisely, noting that $Z$ is affine, we describe the global sections of their cohomology sheaves. Our second goal is to use this description in order to compute their depth at the cone point. We combine all of this in the following:

\begin{theorem}\label{thm:DB-cones}
	With the notation above:
	
	\noindent 
(1) The Du Bois complexes $\DB_Z^k$ are given by 
\[\Gamma (Z, \H^0\DB_Z^0) \simeq \CC\oplus \bigoplus_{m\ge 1}\HH^0(X, \DB_X^0\otimes L^m) \simeq \CC\oplus \bigoplus_{m\ge 1} H^0(X, \H^0\DB_X^0\otimes L^m),\]
	\[\Gamma (Z, \H^i\DB_Z^0) \simeq \bigoplus_{m\ge 1}\HH^i(X, \DB_X^0\otimes L^m) \text{ for $i\ge 1$},\]
	and for $k\ge 1$,
	\[\Gamma (Z, \H^i\DB_Z^k) \cong \bigoplus_{m\ge 1}\HH^i(X,\DB_X^k\otimes L^m)\oplus \bigoplus_{m\ge 1}\HH^i(X, \DB_X^{k-1}\otimes L^m).\]
	
	\noindent 
(2)  We have $\depth_x \DB_Z^k\ge 1$ for every $k\ge 0$. Moreover, if $d\ge 1$ is an  integer:
 
 \smallskip
 \noindent
 (i) ~$\depth_x \DB_Z^0>d$ if and only if
	\begin{itemize}
		\item $\HH^0(X, \DB_X^0\otimes L^m)=0$ for all $m\le -1$;
		\item $\HH^i(X, \DB_X^0\otimes L^m)=0$ for all $m\le 0$ and $1\le i\le d-1$.
	\end{itemize}
	
\noindent
(ii)  If $k\ge 1$, then $\depth_x \DB_Z^k >d$ if and only if
	\begin{itemize}
		\item $\HH^i(X, \DB_X^k\otimes L^m)=\HH^i(X, \DB_X^{k-1}\otimes L^m)=0$ for $m\le -1$ and $0\le i\le d-1$;
		\item $\HH^0(X, \DB_X^k)=0$;
		\item The map $\HH^{i}(X, \DB_X^{k-1})\to \HH^{i+1}(X, \DB_X^k)$ is an isomorphism for $0\le i\le d-2$ and injective for $i=d-1$. 
	\end{itemize}

\end{theorem}

In this section we focus on the proof of Theorem \ref{thm:DB-cones}(1), while (2) will be proved in the next.
To this end, we consider the blow-up square
\[\begin{tikzcd}
	E\simeq X & {\widetilde{Z}} = \mathbf{Spec}_X \big(\bigoplus_{m\ge 0} L^m \big) \\
	x & Z
	\arrow["i", from=1-1, to=1-2]
	\arrow["f"', from=1-1, to=2-1]
	\arrow[from=2-1, to=2-2]
	\arrow["f", from=1-2, to=2-2]
\end{tikzcd}\]
obtained by blowing up $Z$ at the cone point. We denote
$$U := Z \smallsetminus \{x\} = \widetilde{Z} \smallsetminus E.$$

Using Lemma \ref{general-DB}(4), we have the following exact triangles:
\[\DB_Z^0\to \bR f_*\DB_{\widetilde{Z}}^0\oplus \mathcal{O}_{x} \to \bR f_*\DB_E^0\xto{+1}\]
and 
\[\DB_Z^k\to \bR f_*\DB_{\widetilde{Z}}^k\to \bR f_*\DB_E^k\xto{+1}\]
for $k\ge 1$. Since $\widetilde{Z}$ is an $\AA^1$-bundle over $E \simeq X$, these exact triangles are split:
\begin{lemma}
	With the notation above, for each $k \ge 0$ there exists a ``pullback" map
	\[\bR f_*\DB_E^k\to \bR f_*\DB^k_{\widetilde{Z}}\]
	inverse to the natural restriction map $\bR f_*\DB^k_{\widetilde{Z}}\to \bR f_*\DB_E^k$.
\end{lemma}

As a consequence, after passing to hypercohomology, the exact triangles above lead to direct sum decompositions:
\begin{equation}\label{eqn:blow-up1}
\HH^0(\widetilde{Z}, \DB^0_{\widetilde{Z}})\oplus \CC\simeq  \HH^0(Z, \DB_Z^0) \oplus \HH^0(E, \DB_E^0)
\end{equation}
and
\begin{equation}\label{eqn:blow-up2}
 \HH^i(\widetilde{Z}, \DB^k_{\widetilde{Z}}) \simeq \HH^i(Z, \DB_Z^k)\oplus \HH^i(E, \DB_E^k)
 \end{equation}
for $i\ge 0$ and $k\ge 1$. Moreover, the summands $ \HH^i(E, \DB_E^k)$ can be identified with $\HH^i(X, \DB_X^k)$, via the zero section of $\pi$.
We next  compute $\pi_*\DB^k_{\widetilde{Z}}$, via the projection $\pi\colon \widetilde{Z} \to X$.

\begin{lemma}\label{lemma:DB-complex-cone-with-cone-point}
	With the notation above, we have
	\[\pi_*\DB^0_{\widetilde{Z}} \simeq \bigoplus_{m\ge 0} \DB_X^0\otimes L^m,\]
	and for each $k\ge 1$, a split exact triangle
	\[\bigoplus_{m\ge 0} \DB_X^k\otimes L^m\to \pi_*\DB^k_{\widetilde{Z}}\to \bigoplus_{m\ge 1}\DB_X^{k-1}\otimes L^m\xto{+1}.\]
	\begin{proof}
		Recall that $\widetilde{Z}=\mathbf{Spec}_X (\Sym L)$. Given a hyperresolution $\epsilon_{\bullet}: X_{\bullet}\to X$, it induces a hyperresolution $\widetilde{Z}_{\bullet}=\mathbf{Spec}_{X_{\bullet}} (\Sym \epsilon_{\bullet}^*L)$ of $\widetilde{Z}$. For each $\widetilde{Z_i}\to X_i$, we have
		\[\pi_*\O_{\widetilde{Z_i}} \simeq \bigoplus_{m\ge 0} L^m\]
		and for $k\ge 1$, a split exact sequence
		\[0\to \bigoplus_{m\ge 0}\Omega_{X_i}^p\otimes \epsilon_i^*L^m\to \pi_*\Omega^p_{\widetilde{Z}_i}\to \bigoplus_{m\ge 1}\Omega_{X_i}^{p-1}\otimes \epsilon_i^*L^m\to 0.\]
		Since the splitting given by $d$ is compatible with the maps in the hyperresolution, we get the result by pushing these forward to $X$, and using the projection formula.
	\end{proof}
\end{lemma}

We are now ready to deduce the result stated at the beginning of the section:

\begin{proof}[{Proof of Theorem \ref{thm:DB-cones}(1)}]
	Putting together ($\ref{eqn:blow-up1}$), ($\ref{eqn:blow-up2}$), and Lemma \ref{lemma:DB-complex-cone-with-cone-point}, we obtain
	\[\HH^0(Z, \DB_Z^0)\simeq \CC\oplus \bigoplus_{m\ge 1} \HH^0(X, \DB_X^0\otimes L^m)\]
	\[\HH^i(Z, \DB_Z^0) \simeq \bigoplus_{m\ge 1} \HH^i(X, \DB_X^0\otimes L^m)\,\,\,\,{\rm for} \,\,\,\,i \ge 1,\]
	and for each $i$ and $k\ge 1$, a split exact sequence
	\[0\to \bigoplus_{m\ge 1} \HH^i(\DB_X^k\otimes L^m)\to \HH^i(Z, \DB^k_Z)\to \bigoplus_{m\ge 1}\HH^i(\DB_X^{k-1}\otimes L^m)\to 0.\]
	
	Since $Z$ is affine, we have $\Gamma (Z, \H^i \DB_Z^k) = \HH^{i}(Z, \DB_Z^k)$, from which the result is immediate.
\end{proof}

\medskip

It is also worth noting the following quick consequence of Theorem \ref{thm:DB-cones}(1), presumably known to experts:

\begin{corollary}
	The cone $Z=C(X,L)$ is seminormal if and only if $X$ is seminormal.
\end{corollary}	
	\begin{proof}
		It is well known that a variety $X$ is seminormal if and only if the natural map $\O_X\to\H^0\DB^0_X$ is an isomorphism.
		
		In our case, the cone $Z$ is seminormal if and only if
		\[\varphi: \Gamma (Z, \O_Z) = \bigoplus_{m\ge 0} H^0(X, L^m) \to \Gamma (Z, \H^0\DB_Z^0) = \CC\oplus \bigoplus_{m\ge 1}H^0(X, \H^0\DB_X^0\otimes L^m)\]
		is an isomorphism. Hence it is clear that if $X$ is seminormal, then so is $Z$.
		
		Conversely, suppose $Z$ is seminormal. Since the map $\varphi$ preserves the grading given by $L$, we have isomorphisms
		\[H^0(X, L^m)\to H^0(X, \H^0\DB_X^0\otimes L^m)\] 
		for all $m\ge 1$, induced by the inclusion of sheaves $\O_X \hookrightarrow \H^0\DB_X^0$; using Serre's theorems, it is then straightforward to check 
		that this inclusion has to be an isomorphism.
		\end{proof}

Finally, when $L$ is very ample, one can also consider the classical cone $C(X)$ over the corresponding embedding in projective space. 
For future use, we record the following result that relates $C(X,L)$ to $C(X)$.
	
	\begin{lemma}\label{lemma:abstract-vs-classical-cones}
		Suppose $L = \O_X(1)$ is very ample, defining an embedding 
		$X \subset \PP^n$. Then for each $k$,
		\[\DB^k_{C(X)}\to \bR \pi_*\DB^k_{C(X,L)}\]
		is a quasi-isomorphism, where $\pi$ is the natural map $C(X,L)\to C(X)$. Consequently, $\lcdef(C(X))=\lcdef(C(X,L))$.
		\begin{proof}
			The isomorphism $\DB_{C(X)}^k\to \bR \pi_*\DB^k_{C(X,L)}$ follows from Lemma \ref{general-DB}(4), given that $\pi$ is an isomorphism away from the vertex points. The last assertion then follows from Corollary \ref{cor:depth-DB-complex}.

		\end{proof}
	\end{lemma}

\subsection{Depth of Du Bois complexes of cones}
In this section we focus on the proof of Theorem \ref{thm:DB-cones}(2). In other words, we characterize the depth of $\DB_Z^k$, where $Z$ is the cone over a possibly singular projective variety $X$, in terms of cohomological data of the Du Bois complexes of $X$.

To do so, still using the notation of the previous section, we must first also understand the Du Bois complexes of the complement $U = Z \smallsetminus \{x\}$ of the vertex of the cone.  Note first that 
a completely similar argument as in Lemma \ref{lemma:DB-complex-cone-with-cone-point} gives: 

\begin{lemma}\label{lemma:DB-complex-cone-without-cone-point}
	For $U=\widetilde{Z} \smallsetminus E=Z \smallsetminus \{x\}$ we have
	\[\pi_*\DB^0_{U}\simeq \bigoplus_{m\in \ZZ} \DB_X^0\otimes L^m,\]
	and for each $k\ge 1$, a (not necessarily split) exact triangle
	\[\bigoplus_{m\in \ZZ} \DB_X^k\otimes L^m\to \pi_*\DB^k_U\to \bigoplus_{m\in \ZZ}\DB_X^{k-1}\otimes L^m\xto{+1}.\]
	\end{lemma}

\smallskip	

Passing to hypercohomology, we obtain
\[\HH^i(U, \DB^0_U) \simeq \bigoplus_{m\in \ZZ}\HH^i(X, \DB_X^0\otimes L^m)\]
for $i\ge 0$, and for  $k\ge 1$, a long exact sequence
\[\cdots\to \bigoplus_{m\in \ZZ} \HH^i(\DB_X^k\otimes L^m)\to \HH^i(U, \DB^k_U)\to \bigoplus_{m\in \ZZ}\HH^i(\DB_X^{k-1}\otimes L^m)\xto{d_i} \bigoplus_{m\in \ZZ} \HH^{i+1}(\DB_X^k\otimes L^m)\to \cdots\]

\medskip

\smallskip

\begin{lemma}\label{lemma:zero-connecting-morphism}
     The differentials $d_i$ in the long exact sequence above are zero, except on the $m=0$ summand, where 
	\[\HH^i(X,\DB_X^{k-1})\xto{d_i} \HH^{i+1}(X,\DB_X^k)\]
	is induced by cup product with $c_1(\epsilon_j^*L)$ on each term $X_j$ in a hyperresolution of $X$.
		\begin{proof}
		For each $k$ and $m$, completely analogously to Lemma \ref{general-DB}(5), 
		we have a spectral sequence
		\[E_1^{q,p}=H^p (X_q, \Omega_{X_q}^k\otimes \epsilon_q^* L^m)\implies \HH^{p+q}(X, \DB_X^k\otimes L^m),\]
		by means of which the differentials $d_i$, with $i = p +q$,  are induced by the corresponding differentials 
		\[H^p (X_q, \Omega_{X_q}^{k-1}\otimes \epsilon_q^* L^m)\xto{d_p} H^{p+1}(X_q, \Omega_{X_q}^{k}\otimes \epsilon_q^* L^m)\]
		on each term in the hyperresolution.
		
		It suffices therefore to show the statement of the Lemma when $X$ is smooth, meaning that in this case 
		\[d_p\colon H^p (X, \Omega_{X}^{k-1}\otimes L^m)\to  H^{p+1}(X, \Omega_{X}^{k}\otimes L^m)\]
		is $0$ on the terms with $m\neq 0$, and cup product with $c_1 (L)$ on the terms with $m=0$.
		
		To this end, let $\omega\in H^p(X,\Omega_X^{k-1}\otimes L^m)$ be represented by the \v Cech $p$-cocycle \[\{\omega_{i_0i_1\cdots i_p} t_{i_0i_1\cdots i_p}^m\}_{i_0i_1\cdots i_p}\in \check{C}^p(X,\Omega^{k-1}_X\otimes L^m).\]
		where we fix an open cover $\{U_i=\Spec A_i\}_i$ over which $L=t_i A_i$ is locally free, and we denote by $t_{i_0i_1\cdots i_p}:=t_{i_0}|_{U_{i_0}\cap \cdots \cap U_{i_p}}$. To simplify notation, we will denote this restriction by $t_{i_0}$ as well.
		
		If $m\neq 0$, then $\omega$ lifts to a cocycle $\{\frac{1}{m} d(\omega_{i_0i_1\cdots i_p} t_{i_0}^m)\}_{i_0i_1\cdots i_p}\in \check{C}^p(\Omega^k_{\tilde{Z} \smallsetminus E})$, thus its image in $\check{C}^{p+1}(\Omega^k_{\tilde{Z}\smallsetminus E})$ is zero.

		If $m=0$, then $\omega$ lifts to $\{\frac{1}{t_{i_0i_1\cdots i_p}} d(\omega_{i_0i_1\cdots i_p} t_{i_0})\}_{i_0i_1\cdots i_p}\in\check{C}^p(\Omega^k_{\tilde{Z}\smallsetminus E})$. Using the cocycle condition  $\sum_j (-1)^j \omega_{i_0i_1\cdots \hat{i_j}\cdots i_pi_{p+1}} = 0$, one computes that
		\begin{align*}
			\varphi_p(\omega)  
			&=\{\omega_{i_1i_2\cdots i_{p+1}}\wedge \frac{dt_{i_1}}{t_{i_1}}+\sum_{j>0} (-1)^j  \omega_{i_0i_1\cdots \hat{i_j}\cdots i_pi_{p+1}}\wedge \frac{dt_{i_0}}{t_{i_0}}\}_{i_0i_1\cdots i_p i_{p+1}} \\
			&=\{\omega_{i_1i_2\cdots i_{p+1}}\wedge (\frac{dt_{i_1}}{t_{i_1}}-\frac{dt_{i_0}}{t_{i_0}})\}_{i_0i_1\cdots i_p i_{p+1}},
		\end{align*}
		which (up to sign) represents the class $\omega\cup \{d\log g_{ij}\}_{i,j}$, where $g_{ij}$ are the transition maps of $L$, satisfying $t_{i}=g_{ij}t_{j}$. Since the class of $d\log g_{ij}$ can be identified with $c_1 (L) \in H^2(X,\ZZ)\cap H^{1,1} (X)$ (see e.g. \cite[p.141]{GH}), we obtain the conclusion.
\end{proof}
\end{lemma}

\begin{remark}\label{rmk:compatibility}
It is well known that for line bundle $L$ on a (possibly singular) projective variety $X$, we have 
$$c_1 (L) \in F^1 H^2 (X, \CC) = H^2 (X,F^1 \DB_X^\bullet),$$
and therefore it defines an element in $\HH^1 (X, \DB_X^1)$; see e.g. \cite{AK} for a discussion. The differentials in Lemma \ref{lemma:zero-connecting-morphism} can be seen 
as cup product with this element; equivalently, they arise from the action of $c_1 (L)$ on singular cohomology
	$$\cdot \cup c_1 (L) \colon H^{i + k -1} (X, \CC) \to H^{i + k +1} (X, \CC),$$
via Lemma \ref{general-DB}(3). However, Lemma  \ref{lemma:zero-connecting-morphism} can also be taken to be a definition of this action.

When in addition $L$ is very ample, giving an embedding $\iota: X\hookrightarrow \PP^N$, there are induced morphisms of 
mixed Hodge structures $\iota^*_i: H^i(\PP^N,\CC)\to H^i(X,\CC)$, commuting with cup product, and  such that $\iota^*_2 (c_1(\O(1)))=c_1(L)$.
Thus the Lefschetz maps on $\PP^N$ and $X$ are compatible, in the sense that for each $i$ there is a commutative diagram
\[\begin{tikzcd}
	&H^i (\PP^N, \CC)\ar[r,"c_1(\O(1))"]\ar[d] &H^{i+2}(\PP^N, \CC)\ar[d] \\
	&H^i (X, \CC)\ar[r,"c_1(L)"] & H^{i+2 }(X, \CC).
\end{tikzcd}\]
\end{remark}

\medskip

As a consequence of the calculations above and of the results of the previous section, we obtain:

\begin{corollary}\label{cor:technical-isom}
With notation as above, we have:

\noindent
(i)  The maps $\HH^i(Z, \DB_Z^k)\to \HH^i(U, \DB_U^k)$ induced by restriction are injective for all $i,k$. 

\noindent
(ii)  The map $\HH^0(Z, \DB_Z^0)\to \HH^0(U, \DB_U^0)$ is an isomorphism if and only if
	\[\HH^0(X, \DB_X^0\otimes L^m)=0 \,\,\,\,\text{ for all } \,\,m\le -1.\]
For $i\ge 1$, the map $\HH^i(Z, \DB_Z^0)\to \HH^i(U, \DB_U^0)$ is an isomorphism if and only if
	\[\HH^i(X, \DB_X^0\otimes L^m)=0 \,\,\,\,\text{ for all } \,\,m\le 0.\]
	
\noindent
(iii) For $k\ge 1$, the map $\HH^i(Z, \DB_Z^k)\to \HH^i(U, \DB_U^k)$ is an isomorphism if and only if
	\begin{itemize}
		\item $\HH^i(X, \DB_X^k\otimes L^m)=\HH^i(X, \DB_X^{k-1}\otimes L^m)=0$ for $m\le -1$;
		\item  The differential
			\[\HH^{i-1}(X,\DB_X^{k-1})\xto{d_i} \HH^i(X,\DB_X^k)\text{ is surjective}\]
		and the differential
		\[\HH^i(X,\DB_X^{k-1})\xto{d_i} \HH^{i+1}(X,\DB_X^k)\text{ is injective}.\]
	\end{itemize}
\end{corollary}

We are finally ready to compute the depth at the Du Bois complexes of $Z$, which is crucial for the proof of Theorem \ref{thm:main-general}, but may also be of independent interest.

\begin{proof}[{Proof of Theorem \ref{thm:DB-cones}(2)}]
Recall that we denote $U = Z \smallsetminus \{x\}$. For each $k$, there is a long exact sequence
\[\cdots\to \HH^i_x(Z, \DB_Z^k) \to \HH^i(Z, \DB_Z^k) = H^0(Z, \H^i\DB_Z^k) \to \HH^i(U, \DB_U^k)\to \cdots.\]
Thus, the condition \[\depth_x \DB_Z^k>d \iff \HH_x^i(Z, \DB_Z^k)=0 \,\,\,\,\text{ for all }\,\,\,\, i\le d\]
is equivalent to the following two conditions:
\begin{enumerate}
	\item $\HH^i(Z, \DB_Z^k)\to \HH^i(U, \DB_U^k)$ is an isomorphism for $0\le i\le d-1$;
	\item $\HH^d(Z, \DB_Z^k)\to \HH^d (U, \DB_U^k)$ is injective.
\end{enumerate}
The assertion is then an immediate consequence of Corollary \ref{cor:technical-isom}.
\end{proof}

\begin{remark}[{\bf The case when $X$ is smooth}]
We make some remarks about the content of Theorem \ref{thm:DB-cones}(2) when $X$ is smooth, as  a blueprint 
for what we should aim for in the singular case. With this assumption, the statement translates into the fact, say when $k\ge 1$, that 
$\depth_x \DB_Z^k > d$ is equivalent to the following two conditions being satisfied simultaneously:

\smallskip

\noindent
(i) $H^i (X, \Omega_X^k \otimes L^{\otimes m}) = H^i (X, \Omega_X^{k-1} \otimes L^{\otimes m}) = 0$ for all $m \le - 1$ and $0 \le i \le d- 1$.

\smallskip

\noindent
(ii)  The maps $H^{k -1, i} (X) \to H^{k, i +1} (X)$ induced by cup product with $c_1 (L)$ are isomorphisms for $-1 \le i \le d -2$, and injective for $i = d -1$.

If $d \le n - k +2$, then  $i \le d -1$ implies $i + k < n = \dim X$, and therefore condition (i) holds automatically thanks to Nakano vanishing. The injectivity of all the maps in (ii) is also automatic, because of the Hard Lefschetz theorem. These observations and similar ones for $k=0$, in conjunction with Corollary 
\ref{cor:depth-DB-complex} and Remark \ref{rmk:compatibility}, essentially finish the proof of Theorem \ref{thm:main} when $X$ is smooth.

Thus when $X$ is singular we will need in particular a suitable replacement for the Nakano vanishing theorem. This is of independent interest, 
and is the topic of the next section.
\end{remark}

\subsection{Vanishing for dual Du Bois complexes}\label{scn:vanishing}
The Du Bois complexes of a projective variety $X$ of dimension $n$ are known by \cite[Theorem V.5.1]{GNPP} to satisfy the analogue of the Kodaira-Akizuki-Nakano vanishing theorem, 
in the sense that 
$$\HH^q (X, \DB_X^p \otimes L) = 0 \,\,\,\,\,\,{\rm for~all}\,\,\,\,p + q > n,$$
for any ample line bundle $L$ on $X$. The Serre dual statement for twists by $L^{-1}$ does not however hold automatically, as in general we only have 
morphisms $\DB_X^{n-p} \to \DD(\DB_X^p)$ that are not necessarily isomorphisms, where 
$$\DD(-)=\RHom(-,\omega_X^{\bullet})[-n]$$ 
is the (shifted) Grothendieck duality functor. Instead, the correct statement turns out to be the following, which is an expanded version of Theorem \ref{thm:intro-vanishing} in the Introduction.

\begin{theorem}\label{prop:dual-DB-vanishing}
	Let $X$ be a projective variety of dimension $n$, and $L$ an ample line bundle on $X$. Then
	\[\HH^q(X,\DD(\DB_X^p)\otimes L)=0 \quad \text{ for } \,\,\,\,q-p>\lcdef (X).\]
	Equivalently, 
	\[\HH^q(X, \DB_X^p \otimes L^{-1})=0 \quad \text{ for } \,\,\,\, p + q < n - \lcdef (X).\]
		\begin{proof}
		The second statement follows from the first by Grothendieck-Serre duality. 
		
	To check the first statement, fix an embedding $X\hookrightarrow \PP^N$ given by a sufficiently high power of $L$. Let $\QQ_{\PP^N}^H[N]$ denote the trivial Hodge module on $\PP^N$. Then, as in \cite[Proposition 13.1]{MP-LC}, the Du Bois complexes satisfies
		\[\DB_X^p \simeq \RHom_{\O_Y}(\Gr^F_{p-N}\DR_{\PP^N} i_*i^!\QQ_{\PP^N}^H[N],\omega_{\PP^N}^{\bullet})[p- N]\]
		Dualizing, we obtain
		\begin{align*}
			\DD(\DB_X^p) &= \RHom_{\O_X}(\DB_X^p,\omega_X^{\bullet})[-n]\\
			&\cong \RHom_{\O_{\PP^N}}(\DB_X^p,\omega_{\PP^N}^{\bullet})[-n]\\
			&\cong (\Gr^F_{p-N} \DR_{\PP^N} i_*i^! \QQ_{\PP^N}^H[N])[N-p-n].
		\end{align*}
		Consider now the spectral sequence
		\[E_2^{i,j}=\HH^i(X, \Gr^F_{p-N}\DR_{\PP^N}\H^j(i_*i^!\QQ_{\PP^N}^H[N])\otimes L) \implies \]
		\[\implies \HH^{i+j}(X, \Gr^F_{p-N}\DR_{\PP^N} (i_*i^!\QQ_Y^H[N])\otimes L).\]
				
		We have that
		\begin{itemize}
			\item $\H^j(i_*i^!\QQ_{\PP^N}^H [N]) = \H^j_X\O_{\PP^N}=0$ for $j>\lcd(X, \PP^N)$;
			\smallskip
			\item $\HH^i(X,  \Gr^F_{p-N}\DR_{\PP^N}\H^j(i_*i^!\QQ_{\PP^N}^H[N]) \otimes L)$=0 for $i>0$, by Saito's vanishing theorem 
			for mixed Hodge modules \cite[Scn.2.g]{Saito}.
		\end{itemize}
		It follows that $E_2^{i,j}=0$ if $i+j>\lcd(X,\PP^N)$. We conclude that 
		$$\HH^q(X, \DD(\DB_X^p)\otimes L)=0$$ 
		when $q-p+N- n>\lcd(X,\PP^N)$, or equivalently $q-p>{\rm lcdef}(X)$, as required. 
	\end{proof}
\end{theorem}

\subsection{Proof of the main result}
The following statement is the more general version of Theorem \ref{thm:main} promised in the Introduction, that applies to any abstract cone over $X$.

\begin{theorem}\label{thm:main-general}
	Let $X$ be an $n$-dimensional projective variety endowed with an ample line bundle $L$, and let $Z=C(X, L)$ be the associated cone over $X$. Then the following are equivalent:
	\begin{enumerate}
		\item $\lcdef(Z)\le c$
		\item $\lcdef (X)\le c$, and the ``Lefschetz" maps 
		$$H^i (X, \CC)  \to H^{i + 2} (X, \CC) $$ 
		given by cup product with $c_1 (L)$ 
		are isomorphisms for $-1\le i \le n-3 - c$ and injective for $i \le n-2 - c$, with the convention that $H^{-1} (X, \CC) = 0$.\footnote{In particular, all $H^i (X, \CC)$ in this range are $0$ for $i$ odd, and $1$-dimensional for $i$ even.}
\end{enumerate}
\end{theorem}

	\begin{proof}
	First we note the inequality 
	\begin{equation}\label{eqn:lcd-ineq-2}
	\lcdef (X)\le \lcdef (Z).
	\end{equation}
	Indeed, we clearly have  $\lcdef (Z \smallsetminus \{x\}) \le \lcdef (Z)$, where $x \in Z$ is the cone point, while 
	$\lcdef (Z \smallsetminus \{x\}) = \lcdef (X)$, since $Z \smallsetminus \{x\}$ is a locally trivial $\CC^*$-bundle over $X$.
	
	Note now that by Corollary \ref{cor:depth-DB-complex}, we have
		\[{\rm lcdef}(Z)=\dim Z- \min_{k\ge 0} ~\{\depth \DB_Z^k + k \}.\]
	As a consequence, we have an equivalence between the following two conditions:
	\begin{enumerate}
		\item $\lcdef(Z)\le c$
		\item $\lcdef (X)\le c$, and 
		\[\depth_x \DB_Z^k\ge n+1-c-k\]
		holds for all $k \ge 0$.
	\end{enumerate}
Hence we focus on understanding when this last condition holds.
		By Theorem \ref{thm:DB-cones}(2), applied with $d=n-c-k$,\footnote{Note that the Proposition applies when $d\ge 1$, or equivalently $k\le n-c-1$.} it is in turn equivalent to the following conditions being satisfies simultaneously:
		\begin{enumerate}
			\item[{(a)}] $\HH^i(X, \DB_X^0)=0$ for $1\le i\le n-c-1$;
			\item[{(b)}] $\HH^0(X, \DB_X^k)=0$ for $1\le k\le n-c-1$;
			\item[{(c)}] $\HH^i(X, \DB_X^k\otimes L^m)=0$ for all $m\le -1$ and $i+k\le n-c-1$;
			\item[{(d)}] For $1\le k\le n-c-2$, the map $\HH^i(X, \DB_X^{k-1})\to \HH^{i+1}(X, \DB_X^k)$ is isomorphism when  $0\le i\le n-c-k-2$, and injective when $i=n-c-k-1$.
		\end{enumerate}
		The main point is now to note that by Theorem \ref{thm:intro-vanishing},  condition (c) is guaranteed to hold 
precisely when $\lcdef (X) \le c$.
		
On the other hand, condition (a), (b) and (d) taken together are easily seen to be equivalent to
	\begin{enumerate}
		\item[{(d$^{\prime}$)}]  For all $i,k\ge 0$, the map 
			$$\HH^i(X, \DB_X^k)\to \HH^{i+1}(X, \DB_X^{k+1})$$
			is an isomorphism when $-1 \le i+k\le n-c-3$
			and injective when $i+k=n-c-2$.	
\end{enumerate}
This in turn is equivalent to the condition on the Lefschetz maps in (2) in the statement of the Theorem, using Lemma \ref{general-DB}(3)
 and the fact that these maps are compatible with the differentials in Lemma \ref{lemma:zero-connecting-morphism}; see also Remark \ref{rmk:compatibility}.
\end{proof}
	
\begin{proof}[Proof of Theorem \ref{thm:main}]
In the classical setting, we have in addition that $X$ is embedded in some $\PP^N$, and $L = \O_X (1)$. By Lemma \ref{lemma:abstract-vs-classical-cones}, it suffices to prove Theorem \ref{thm:main} for $Z=C(X,L)$. 
The point is simply to observe that the Hodge-theoretic conditions in (2) in Theorem  \ref{thm:main-general} are in this case equivalent to the condition that 
the restriction maps are isomorphisms 
\begin{equation}\label{eqn:WL}
H^i (\PP^N, \CC) \overset{\simeq}{\longrightarrow} H^i (X, \CC) \,\,\,\,\,\,{\rm for ~all} \,\,\,\, i \le n - 1 - c.
\end{equation}

Note that the Lefschetz maps on $\PP^N$ and $X$ are compatible, as in the diagram in Remark \ref{rmk:compatibility}. Recall that the second condition in (2) in Theorem \ref{thm:main-general} says that the maps
$$H^i (X, \CC) \to H^{i+2} (X, \CC),$$
given by cup product with $c_1 (L)$, are isomorphisms for all $-1 \le i \le n-3 - c$, and injective for $i \le n-2 - c$.  This condition certainly  holds if  ($\ref{eqn:WL})$ is satisfied. (Injectivity holds for $i = n - 2 - c$, since in this case $H^i (X, \CC)$ must be $0$ or $1$-dimensional, while the Lefschetz map is nontrivial if $H^i (\PP^N, \CC)$ is so.)

Conversely, since $H^1 (X, \CC) = 0$, assuming that the bottom horizontal maps in the diagram are isomorphisms for $i \le n - 3 - c$, it is clear that  ($\ref{eqn:WL}$) holds.
\end{proof}

\section{Other applications}

Here we describe further consequences of our calculation of the Du Bois complexes of cones over singular varieties.

\subsection{Higher Du Bois singularities and Bott vanishing}
A natural strengthening of the notion of Du Bois singularities has been considered in \cite{MOPW}, \cite{JKSY}, \cite{MP-LC} and \cite{MP2} in the case of local complete intersections. The weaker condition of \emph{pre-$p$-Du Bois} singularities was defined in general in \cite{SVV}; it means that up to degree $p$ the Du Bois complexes have no higher cohomologies:
$$\DB_X^j \simeq \H^0 \DB_X^j \,\,\,\,\,\,{\rm for~all}\,\,\,\, j \le p.$$ 

Using Theorem \ref{thm:DB-cones}(1), we obtain a criterion for cones over singular varieties to be pre-$p$-Du Bois, extending \cite[Corollary 7.4]{SVV}.

\begin{corollary}\label{cor:pDB}
	Let $X$ be a projective variety endowed with an ample line bundle $L$, and let $Z=C(X,L)$ be the associated cone over $X$. Then $Z$ has pre-$p$-Du Bois singularities if and only if
	\[\HH^i(X, \DB_X^j \otimes L^m)=0 \quad \text{ for all } \,\, i,m\ge 1,~ j \le p.\]
\end{corollary}

Following \cite[Section 2.3]{KT}, we say that a projective variety $X$ \emph{satisfies Bott vanishing} if 
\[H^i(X, \Omega_X^{[j]}\otimes A)=0\]
for all $j\ge 0$ and $i >0$, and any ample line bundle $A$ on $X$. Here  $\Omega_X^{[j]}$ denotes the reflexive differentials 
$(\Omega_X^{j})^{\vee \vee}$.

\begin{corollary}
Let $X$ be a projective variety with rational singularities. Then the following are equivalent:
\begin{enumerate}
\item $Z = C(X, L)$ has pre-$p$-Du Bois singularities for all ample line bundles $L$ and all $p$. 
\item $X$ has pre-$p$-Du Bois singularities for all $p$ and satisfies Bott vanishing.
\end{enumerate}
\end{corollary}
\begin{proof}
Due to the main result of \cite{KS}, when $X$ has rational singularities we have $\H^0 \DB_X^j \simeq \Omega_X^{[j]}$ for all $j$.
Thus (1) follows from (2) by Corollary \ref{cor:pDB}. Conversely, suppose first that $X$ does not have pre-$p$-Du Bois singularities, meaning
$\H^i\DB_X^j\neq 0$ for some $i>0$ and $j \le p$. By Serre's theorems, we can choose $L$ sufficiently positive, so that $H^s(X,\H^t\DB_X^j\otimes L)=0$ for all $s>0$ and $t\ge 0$, and $H^0(X,\H^i\DB_X^j\otimes L)\neq 0$. Then, the spectral sequence
\[E_2^{s,t}=H^s(X,\H^t\DB_X^j\otimes L)\implies \HH^{s+t}(X,\DB_X^j\otimes L)\]
degenerates, and gives
\[\HH^i(X, \DB_X^j \otimes L)\simeq H^0(X,\H^i\DB_X^j\otimes L)\neq 0.\]
This is a contradiction, again by Corollary \ref{cor:pDB}, which also implies that $X$ satisfies Bott vanishing once we know 
that it is pre-$p$-Du Bois. 
\end{proof}

\begin{remark}
For varieties whose singularities are not rational, it may in fact make sense to define the Bott vanishing condition in terms of $\H^0 \DB_X^j$ rather than
$\Omega_X^{[j]}$. The statement of the Corollary, minus the rational singularities assumption, would remain unchanged. 

For example, one can show along the lines of \cite{KT} that if $X$ is a variety with int-amplified endomorphisms, then for every ample line bundle $A$ on $X$, we have
\[H^i(X,\H^j\DB_X^k\otimes A)=0\]
for $i>0$ and $j,k\ge 0$.
\end{remark}

\subsection{$K$-groups of cones}\label{scn:K}
Another consequence of Theorem \ref{thm:DB-cones} is the computation of the non-positive $K$-groups of abstract cones over a projective variety in terms of its Du Bois complexes, following the approach in \cite{CHWW}  for classical cones over smooth varieties, and using the main technical result of that paper. Please see the introduction of \emph{loc. cit.} for a discussion of the problem, and of the important techniques introduced by the authors in a series of papers dedicated to the study of $K$-groups.

As in \cite{CHWW-Bass}, for any subfield $F\subset \CC$ and complex variety $X$, one can define the $p$-th Du Bois complex of $X$ over $F$ as
\[\DB_{X/F}^p=\bR \epsilon_{\bullet*}\Omega^p_{X_{\bullet}/F}\]
where $\epsilon_{\bullet}: X_{\bullet}\to X$ is a hyperresolution of $X$. The formula for the Du Bois complexes of cones in Theorem \ref{thm:DB-cones} remains valid in this setting, with the same proof where $\DB_{X}^p$ is replaced by $\DB_{X/F}^p$. Moreover, 
one checks exactly as in the slightly more restrictive \cite[Lemma 2.1]{CHWW-Bass} that  there is an isomorphism
\[ \DB_{X/F}^p \simeq \Omega^p_{cdh,X/F},\]
where  $\Omega^p_{cdh,X/F}:=\bR a_* a^* \Omega^p_{X/F}$ is the sheafification of $\Omega_{X/F}^p$ in the cdh-topology, the map $a: X_{cdh}\to X_{zar}$ being the natural change-of-topology morphism.

\medskip

We are now ready to describe the calculation of non-positive $K$-groups of cones.

\begin{proof}[{Proof of Corollary \ref{cor:K-groups}}]
By \cite[Thereom 1.2]{CHWW},  for the affine cone $Z=C(X,L)$ over the $n$-dimensional projective variety $X$, we have

\begin{itemize}
	\item $K_0(Z) \simeq \ZZ\oplus \bigoplus_{i = 1}^{\dim Z-1} \Gamma(\H^i\DB_{Z/\QQ}^i)/d\big(\Gamma(\H^i\DB_{Z/\QQ}^{i-1})\big)$;
	\smallskip
	\item $K_{-\ell}(Z) \simeq \Gamma (Z, \H^{\ell} \DB_{Z/\QQ}^0) \oplus \bigoplus_{i=1}^{n-\ell} \Gamma (Z, \H^{\ell+i}\DB_{Z/\QQ}^i)/d\big(\Gamma (Z, \H^{\ell+i}\DB_{Z/\QQ}^{i-1})\big)$,  $\ell \ge 1$.
\end{itemize}
Here $d$ is the action on global sections of the differential $\H^{\ell+i}\DB_{Z/\QQ}^{i-1} \to \H^{\ell+i}\DB_{Z/\QQ}^{i}$.

Using the calculation of Du Bois complexes of abstract cones in Theorem \ref{thm:DB-cones}(1)  (in place of 
\cite[Corollary 2.10]{CHWW}) then yields
\begin{itemize}
	\item $K_0(Z)\simeq \ZZ\oplus \bigoplus_{i=1}^n \bigoplus_{m\ge 1} \HH^i(X, \DB_{X/\QQ}^i\otimes L^m)$;
	\smallskip
	\item $K_{-\ell}(Z)\simeq \bigoplus_{i=0}^{n-\ell}\bigoplus_{m\ge 1} \HH^{\ell+i}(X,\DB_{X/\QQ}^i\otimes L^m)$, $\ell \ge 1$.
\end{itemize}

Note that if $X$ is embedded in $\PP^N$ by $L = \O_X(1)$, one can consider the classical cone $C(X)$, and the method in 
\textit{loc. cit.} applies to this as well. Thus for $\ell \ge 1$ we have 
$$K_{-\ell}(C(X))=K_{-\ell}(C(X,L))$$ 
thanks to Lemma \ref{lemma:abstract-vs-classical-cones}. For $\ell=0$, we have
\[K_0(C(X))\simeq \ZZ\oplus \Pic(C(X))\oplus \bigoplus_{i=1}^n \bigoplus_{m\ge 1} \HH^i(X, \DB_{X/\QQ}^i\otimes L^m)\] 
where $\Pic(C(X)) \simeq R^+ / R$; here $R$ is the homogeneous coordinate ring of $X$, and 
$$R^+ \simeq \bigoplus_{m\ge 0} \HH^0(X, \DB_X^0\otimes L^m)$$
is its seminormalization. The last statement follows from Theorem \ref{thm:DB-cones}(1) and the well-known fact that the structure sheaf 
$\O_{Z^{\rm sn}}$ of the seminormalization of $Z$ can be identified with $\H^0 \DB_Z^0$.
For the statement about $\Pic( C(X))$ we are using \cite[Proposition 1.5]{CHWW}.
\end{proof}


\begin{remark}
	As a consequence, $K_{-\ell} (Z) = 0$ for $l> n +1$, where as usual $n = \dim X$. This vanishing result holds in greater generality, see \cite[Theorem 6.2]{CHSW} and \cite[Theorem B]{KST}.  Moreover, if $X$ is a complex algebraic variety with Du Bois singularities, then
	$$K_{-n -1} (Z) \simeq \bigoplus_{m=1}^{\infty} H^{n}(X, L^m),$$
	as in the case when $X$ is smooth. More generally, if the complex algebraic variety $X$ has pre-$p$-Du Bois singularities (see the previous section), the formula for $K_{-\ell} (Z)$ essentially coincides with that in the smooth case for $\ell \ge n-p$, since in this case we have that $ \H^0 \DB_{X/\QQ}^k \simeq  \DB_{X/\QQ}^k$ for $k \le p$; 
	indeed, this is the case by definition for $\DB_X^k = \DB_{X/\CC}^k$, and the statement over $\QQ$ follows inductively from this via exact triangles of the type
\[\Omega_{\CC/\QQ}^1\otimes_{\QQ} \DB_{X/\QQ}^{p-1}\to \DB_{X/\QQ}^p\to \DB^p_{X/\CC}\xto{+1}\]
obtained from the corresponding exact sequences at the level of K\" ahler differentials on a hyperresolution.	
	Note that when $X$ has rational singularities we have $ \H^0 \DB_X^k \simeq \Omega_X^{[k]}$ for all $k$, while if $X$ is a $p$-Du Bois local complete intersection, then $ \H^0 \DB_X^k \simeq \Omega_X^k$ for $k \le p$; see \cite{MP-LC} for a general discussion.
\end{remark}

\begin{remark}[{\bf Higher $K$-groups}]
With the same notation as above, 
according to \cite[Theorem 1.13]{CHWW}, for $l\ge 1$ the group $K_l(Z)$ breaks up into a direct sum of eigenspaces of the Adams operator, given by:
\begin{itemize}
	\item $K_l^{(i)}(Z)\cong HC_{l-1}^{(i-1)}(R)$ for $0<i<l$;
	\item $K_l^{(l)}(Z)\cong K_l^{(l)}(\CC) \oplus \tors \Omega_{Z/\QQ}^{l-1}/d\tors \Omega_{Z/\QQ}^{l-2}$;
	\item $K_l^{(l+1)}(Z)\cong \coker(\Gamma(Z,\H^0\DB_{Z/\QQ}^{l-1})\xto{d} \Gamma(\H^0\DB_{Z/\QQ}^l/\Omega_{Z/\QQ}^l))$;
	\item $K_l^{(i)}(Z)\cong \coker(\Gamma(\H^{i-l-1}\DB_{Z/\QQ}^{i-2})\xto{d}\Gamma(\H^{i-l-1}\DB_{Z/\QQ}^{i-1}))$ for $i\ge l+2$. 
\end{itemize}

For $i\ge l+2$, Theorem \ref{thm:DB-cones}(1) and a similar argument as in the computation of non-positive $K$-groups gives
\[K_l^{(i)}(Z)\simeq \bigoplus_{m\ge 1}\HH^{i-l-1}(X,\DB_{X/\QQ}^{i-1}\otimes L^m).\]
Still using Theorem \ref{thm:DB-cones}, further explicit calculations can be performed for special choices of $X$; these will appear later.
\end{remark}


\end{document}